\documentclass[12pt]{extarticle}
\usepackage{amsmath, amsthm, amssymb, hyperref, color}
\usepackage{graphicx}
\usepackage{caption}
\usepackage{subcaption}
\usepackage{mathtools}
\usepackage{enumerate}
\usepackage[all]{xypic}
\usepackage{verbatim}
\usepackage{tikz}
\usepackage[lined,commentsnumbered,ruled]{algorithm2e}
\usepackage{caption}
\usepackage{subcaption}
\oddsidemargin \evensidemargin
\usepackage{makecell}
\usepackage{array}
\usepackage{todonotes}
\usepackage{wasysym}

\newtheorem{theorem}{Theorem}
\newtheorem{proposition}[theorem]{Proposition}
\newtheorem{lemma}[theorem]{Lemma}
\newtheorem{corollary}[theorem]{Corollary}
\newtheorem{conjecture}[theorem]{Conjecture}
\theoremstyle{definition}

\newtheorem{problem}[theorem]{Problem}
\newtheorem{remark}[theorem]{Remark}

\newtheorem{example}[theorem]{Example}

\newcommand{\PP}{\mathbb{P}}
\newcommand{\WP}{\mathbb{WP}}
\newcommand{\RR}{\mathbb{R}}

\newcommand{\CC}{\mathbb{C} }

\title{\textbf{Recovery from Power Sums}}

\author{Hana Mel\'anov\'a\footnote{Supported by START-Prize Y-966 of the Austrian Science Fund}\;, Bernd Sturmfels and Rosa Winter}
\date{}
\begin{document}

\maketitle
\begin{abstract}
\noindent
We study the problem of recovering a collection of $n$ numbers from
the evaluation of $m$ power sums. This yields a system of polynomial equations,
which can be underconstrained ($m < n$), square ($m=n$), 
or overconstrained $(m>n)$.  
Fibers and images of power sum maps are explored in all three regimes,
and  in settings that range from complex and projective to real and positive.
This involves surprising deviations from the B\'ezout bound, 
and the recovery of vectors from length measurements by
 $p$-norms.
\end{abstract}

\section{Introduction}
This article offers a case study in solving systems of polynomial equations. Our model setting reflects applications of nonlinear algebra in engineering,
notably in signal processing \cite{Tsa},
 sparse recovery \cite{JLM}, and low rank recovery \cite{KKK}.
Suppose there is a secret list of complex numbers $z_1,z_2,\ldots,z_n$. Our task is to find~them.
Measurements are made by evaluating the $m$ powers sums
$  c_j = \sum_{i=1}^n z_i^{a_j}  $,
where $\mathcal{A} = \{a_1,a_2,\ldots,a_m\}$~is
a set of $m$ distinct positive integers.
Our aim is to recover the multiset $z = \{z_1,\ldots,z_n\}$ from
the vector $c=(c_1,\ldots,c_m)$.

To model this problem, for any given  pair $(n,\mathcal{A})$, we consider the polynomial map
\begin{equation}
\label{eq:map}
\phi_{\mathcal{A,\CC}} \,:\, \CC^n \rightarrow \CC^m ,
\quad \hbox{where}\,\,\,
\phi_j \,= \,x_1^{a_j} + x_2^{a_j} + \cdots +  x_n^{a_j}  \qquad
\hbox{for} \,\, j=1,2,\ldots,m.
\end{equation}
We are interested in the image and the fibers of the map 
$\phi_{\mathcal{A},\CC}$. The study of these complex algebraic varieties addresses
the following questions:  Is recovery possible? Is recovery unique? 
This problem is  especially interesting
when $z_1,z_2,\ldots,z_n$ are real, or even positive.
Hence, we also study the maps
$\phi_{\mathcal{A},\RR}$ and $\phi_{\mathcal{A},\geq 0}$
that are obtained by restricting $\phi_{\mathcal{A},\CC}$
to $\RR^n$ and $\RR_{\geq 0}^n$, respectively.
For any of these, we study the following system of $m$ equations
in $n$~unknowns:
\begin{equation}
\label{eq:system}
\phi_{\mathcal{A},\bullet}(x) \,\,= \,\, c .
\end{equation}

There are  three different regimes.
If $m > n$ then (\ref{eq:system}) is overconstrained and has~no solution, unless $c = \phi_{\mathcal{A},\CC}(z)$ for some $z\in\CC^n$, and we anticipate
unique recovery of  $\{z_1,\ldots,z_n\}$.
If $m=n$ then (\ref{eq:system}) is expected to have finitely many solutions,
at most the B\'ezout number $a_1 a_2 \cdots a_n$.
If $m < n$ then
% is underconstrained and
the solutions to  (\ref{eq:system})  form a variety of expected dimension $n-m$.

\begin{example}[$n=3$]
\label{ex:regimes} We illustrate the three regimes. Consider the multiset $z = \{6,8,13\}$. We first allow $m=4$ measurements, with $\mathcal{A} = \{2,5,7,8\}$. Then the system  (\ref{eq:system}) equals
\begin{equation}
\label{eq:system4} \begin{matrix}
x_1^2+x_2^2+x_3^2 &=& 269 ,& &\phantom{oooooo}&  x_1^5+x_2^5+x_3^5 &=& 411837,  \smallskip \\
 x_1^7+x_2^7+x_3^7&=&65125605, &  &\phantom{oooooo} &  x_1^8+x_2^8+x_3^8&=&834187553 .
 \end{matrix}
\end{equation}
For the lexicographic term order with $x_1 > x_2 > x_3$, we compute the reduced Gr\"obner basis
$$ \bigl\{\,
x_1+x_2+x_3-27\,,\,\, x_2^2+x_2x_3+x_3^2-27 (x_2+x_3)+230\,,\, \, (x_3-6)(x_3-8)(x_3-13) \,\bigr\}.$$
This is a $0$-dimensional radical ideal, having six zeros, so
$z = \{6,8,13\}$ is recovered uniquely.

We next take $m=3$ with $\mathcal{A} = \{2,5,7\}$. Here, we 
solve the first three equations in (\ref{eq:system4}).
This square system has $66$ complex solutions,
four less than the B\'ezout number $70 = 2 \times 5 \times 7$.
Finally, we allow only $m=2$ measurements, with $\mathcal{A} = \{2,5\}$.
The first two equations in~(\ref{eq:system4}) define a curve
of degree $10 = 2 \times 5$ in $\CC^3$. Its closure in $\PP^3$ is
a singular curve of genus~$14$.
\end{example}

\begin{remark}
In applications, noise in the data is a  concern. This makes our problem
interesting even for $\mathcal{A} = \{1,2,\ldots,m\}$.
The recent article \cite{Tsa} studies
reliable recovery from noisy power sums in that  case.
However,
from the perspectives of algebraic geometry and exact computations,
the dense case  is not interesting. The power sums reveal the
 elementary symmetric functions,
via Newton's identities,
and hence, our recovery problem amounts to finding the roots of a polynomial of degree $n$ in one variable. For related work see \cite{BCW21}. 
\end{remark}

In this paper,  $\mathcal{A}$ is any set of $m$ distinct positive integers.
Our presentation is organized as follows. In Section \ref{sec:fibers} we show that, for $m\leq n$, the fiber of $\phi_{\mathcal{A},\CC}$ above a generic point in $\CC^m$ has 
the expected dimension $n-m$.  For $m > n$ we expect
the recovery of complex multisets from power sums to be unique when $\gcd(a_1,\ldots,a_m)=1$.
This is stated in Conjecture \ref{conj.:injectivity over CC}.
In Section \ref{sec: square systems}, we study the case $m=n$.
We propose a formula for the number of solutions of (\ref{eq:system}).
This number is generally less than the B\'ezout number $a_1 a_2 \cdots a_n$.
For instance, in Example \ref{ex:regimes}, the drop is from $70$ to $66$.
We shall explain this.
This issue is closely related to the question, put forth in
\cite{CKW}, for which sets $\mathcal{A}$ the
power sums form a regular sequence.
We present supporting evidence for the
conjectures made in \cite{CKW} and we offer generalizations.

In Section \ref{sec3} we turn to the images of the power sum maps
$\phi_{\mathcal{A},\CC}, \phi_{\mathcal{A},\RR} $ and
$\phi_{\mathcal{A},\geq 0}$.
The image of  $\phi_{\mathcal{A},\CC}$
is constructible and has the expected dimension ${\rm min}(m,n)$,
 but it is generally not closed in $\CC^m$.
In the overconstrained case
$(m>n)$, we study the degree and equations of the closure of the image.
For instance,  the image of $\phi_{\mathcal{A},\CC} : \CC^3 \rightarrow \CC^4$
in Example \ref{ex:regimes}  is 
defined by a polynomial of degree $45$ with $304$ terms.
The image of the real map $\phi_{\mathcal{A},\RR} $ is semialgebraic in $\RR^m$.
It is closed if some $a_i$ is even. Moreover, the orthant
$\RR^n_{\geq 0}$ is mapped to a closed subset of $\RR^m_{\geq 0}$.
It is a challenging task is to find a semi-algebraic description of
the image. We take first steps by exploring its algebraic boundary. Delineating the real image involves
the ramification locus in $\CC^n$ and its image in $\CC^m$,
which is the branch locus of $\phi_{\mathcal{A},\CC}$.

In Section \ref{sec: recovery from norms}, we examine our problem over
  the positive real numbers. Here, the recovery from power sums is equivalent to recovery from 
length  measurements by various $p$-norms. This enables a better understanding of the map $\phi_{\mathcal{A},\geq0}$. We prove that recovery is unique in the square case $n=m$, see Proposition \ref{prop: unique recovery of norms}. The image of $\phi_{\mathcal{A},\geq0}$
is expressed as a compact subset in the  probability simplex $\Delta_{m-1}$.
Theorem \ref{thm:waves} characterizes the structure of this set.

\section{Fibers}\label{sec:fibers}

Consider the map $\phi = \phi_{\mathcal{A},\CC}$
from $\CC^n$ to $\CC^m$ whose coordinates are
$\phi_j = \sum_{i=1}^n x_i^{a_j}$.
In this section we examine the fibers of $\phi$
and we  show that they have the expected dimension.
We conclude with a discussion concerning the uniqueness of recovery in the case $m=n+1$. 

Given a point $c = (c_1,\ldots,c_m)$ in $\CC^m$, the defining ideal of the fiber 
$\phi^{-1}(c)$ equals
  $$ I_c \,\, =\,\,
\langle \,\phi_1(x) - c_1 ,\ldots, \phi_m(x)-c_m\, \rangle \,\, \subset \,\,
\CC[x_1,\ldots,x_n]. $$
Our recovery problem amounts to computing the variety 
$V(I_c) = \phi^{-1}(c)$ defined by $I_c$ 
in $\CC^n$.

\begin{proposition}\label{lem: complete intersection}
Assume $m \leq n$. Then the following hold: 
\begin{itemize}
\item[](i) The map $\phi$ is dominant, i.e., the image of $\phi$ is dense in $\CC^m$.
    \item[](ii) For generic $c$, the ideal $I_c$ is radical,
 and its variety $V(I_c)$
 has dimension $n-m$.
\end{itemize}
\end{proposition}

\begin{proof}
The fiber of $\phi$ above a point $c $ is the
variety $V(I_c) \subset \CC^n$. By  \cite[\href{https://stacks.math.columbia.edu/tag/054Z}{Lemma 054Z}]{stacks-project}, the fibers of $\phi$ are generically reduced. This implies that $I_c$ is radical for all points $c$ outside a proper closed subset of $\CC^m$. 
The Jacobian of the map $\phi$ is the  $m \times n$ matrix
\begin{equation}
\label{eq:jacobian} \mathcal{J} \,\,\, = \,\,\,\, \begin{pmatrix}
\displaystyle \frac{\partial \phi_j}{\partial x_i}
\end{pmatrix}_{\substack{1\leq i \leq n\\1\leq j\leq m}} \,\, = \,\,\,\,
\begin{pmatrix}
\displaystyle 
a_{j} x_i^{a_{j}-1}
\end{pmatrix}_{\substack{1\leq i \leq n\\1\leq j\leq m}}. 
\end{equation}
Up to multiplication by a positive integer, each $m \times m$ minor of this matrix is the product of a Vandermonde determinant
and a Schur polynomial; see (\ref{eq:maximal minors}) below.
In particular, none of these
minors of $\mathcal{J}$ is identically zero. Thus, the Jacobian matrix $\mathcal{J}$ has rank~$m$ over
the field $\CC(x_1,\ldots,x_m)$. By \cite[I.11.4]{Le53}, this implies that
the polynomials $\phi_1,\ldots,\phi_m$ are algebraically independent over $\CC$.
From this we conclude that the associated ring homomorphism $$\,\phi^* \,\colon \,\CC[y_1,\ldots,y_m]\rightarrow\CC[x_1,\ldots,x_n], \,\,
y_i \mapsto \phi(x)\,$$ is injective. Hence, our map $\phi$ is dominant, by \cite[Lemma 0CC1]{stacks-project}.
The statement in (i) that the image is dense refers  either to the
Zariski topology or to the classical topology. Both have the same closure in this
situation, by \cite[Corollary 4.20]{MS}.
   
    It now follows from \cite[Theorem 9.9 (b)]{milneAG} that, for all points $c$ outside a
   proper Zariski closed subset of $\CC^m$, the fiber $\phi^{-1}(c)$ has dimension $n-m$. This finishes the proof. 
\end{proof}

The condition that the point $c$ is generic is crucial in Proposition \ref{lem: complete intersection}. The following example shows that the fiber dimension can jump up for special
points $c \in \CC^m$.

\begin{example}[$n=3$] Let $m=3$ and  $\mathcal{A}=\{3,5,7\}$. The generic fiber of the map
$\phi$ consists of $60$ points in $\CC^3$. Interestingly, that number would increase to $66$
if $3$ were replaced with $2$ in $\mathcal{A}$, by Example~\ref{ex:regimes}.
Now, consider the fiber over $c=(0,0,0)$. We examine the homogeneous
ideal $I_{0} = \langle x_1^a + x_2^a + x_3^a \,:\, a \in \mathcal{A} \rangle$.
This  defines three lines of multiplicity three, with an embedded point at the origin. The radical
of this ideal equals $\,
\langle x_1+x_2, x_3 \rangle \cap 
\langle x_1+x_3, x_2 \rangle \cap 
\langle x_2+x_3, x_1 \rangle $.
\end{example}

Let us assume $m>n$, so we are in the overconstrained case.
The following statement is derived from the $m=n$ case in Proposition \ref{lem: complete intersection}, namely by adding additional constraints:

\begin{corollary}
For $m>n$, the fiber of $\phi$ above a generic point in $\CC^m$ is empty. The closure of the image of $\phi$ is an irreducible variety of dimension $n$ in $\CC^m$. The same holds over $\RR$.
\end{corollary}

Describing the image of $\phi$ will be our topic in Section \ref{sec3} and \ref{sec: recovery from norms}.
A generic point $c$ in that image can be created easily, namely by 
setting $c = \phi(z)$ where $z=(z_1,\ldots,z_n)$ is any generic point in $\CC^n$.
We are interested in the fiber $\phi^{-1}(c)$ over such a point $c$. By construction,
that fiber is non-empty: it contains all $n!$ points that are obtained from $z$
by permuting coordinates. For the remainder of this section, assume $\gcd(a_1,\ldots,a_m)=1$. 
Then we conjecture that there are no other points in that fiber.
This would mean that the set $\{z_1,\ldots,z_n\}$  can be recovered uniquely from
any $m$ of its power sums, provided $m \geq n+1$.

\begin{conjecture}\label{conj.:injectivity over CC}
The recovery of a set of $n$ complex numbers from $n+1$ power sums with coprime powers is unique.  
To be precise, for $m = n+1$, the map $\phi$ is generically injective. This means that,
for generic points $z \in \CC^n$, the fiber
$\phi^{-1}(\phi(z))$ coincides with the set of $n!$
coordinate permutations of $z$.
\end{conjecture}

We are also interested in the following more general  conjecture.
Let $\tau = (\tau_1,\ldots,\tau_n)$ be in $ \RR^n_{> 0}$  and consider the map
$\psi: \CC^n \rightarrow \CC^m$, where $\psi_j = \sum_{i=1}^n \tau_i x_i^{a_j}$.
Let ${\rm Stab}(\tau)$ be the subgroup of the symmetric group $S_n$
consisting of all coordinate permutations that fix~$\tau $.

\begin{conjecture}\label{conj:injectivity2}
For generic points $z \in \CC^n$, 
the fiber $\psi^{-1}(\psi(z))$ is precisely the set of all
coordinate permutations of $z$.
The cardinality of this set is equal to $|{\rm Stab}(\tau)|$.
\end{conjecture}

By computing Gr\"obner bases,
we confirmed Conjectures \ref{conj.:injectivity over CC} and \ref{conj:injectivity2} 
 for a range of small~cases.

\section{Square Systems}\label{sec: square systems}

We  here fix $n=m$, so we study the square case.
By Proposition  \ref{lem: complete intersection},
our system (\ref{eq:system}) has~finitely many solutions in $\mathbb{C}^n$.
  Our aim is to find their number. We study this for $n=2$ (Proposition~\ref{prop:n=2}) and
   $n=3$ (Conjecture \ref{conj: number of solutions}).
  This generalizes a conjecture of Conca, Krattenthaler and Watanabe \cite[Conjecture 2.10]{CKW}.
  We conclude with a discussion of the general case $n \geq 4$. 
 
 \smallskip
 
 Our point of departure is a result which links
Proposition \ref{lem: complete intersection}  with B\'ezout's Theorem.

\begin{proposition}
For general measurements $c \in \CC^n$, 
the square system (\ref{eq:system}) has finitely many complex solutions $x \in \CC^n$.
The number of these solutions is bounded above by $a_1a_2 \cdots a_n$.
\end{proposition}

We now define the {\em homogenized system} (HS) to be the system (\ref{eq:system}), where $c_j$ is replaced by $c_j x_{0}^{a_j}$. Note that (HS) has its solutions in $\PP^n$.
What we are interested in for our recovery problem are
 the solutions that do not lie in the hyperplane at infinity $\{x_0 = 0\}$.
Next, we define the {\em system at infinity} (SI) to be (\ref{eq:system}) with $c= 0$.
The solutions of (SI) are in $\PP^{n-1}$. The cone over that projective scheme is the zero fiber of 
the map $\phi$.
We will use the notations (HS) and (SI) both for the systems of equations and the projective schemes defined by them. 

\begin{remark}
The scheme (HS) is in general not the projective closure of the affine part defined by~(\ref{eq:system}), as it can contain higher-dimensional components. For example, set $n=m=4$, and let $\mathcal{A}$ 
consist of four odd coprime integers. The variety in $\CC^4$ defined by the system (\ref{eq:system}) is zero-dimensional by Lemma \ref{lem: complete intersection}. However, the scheme
 (HS) is not zero-dimensional in $\PP^4$, since it contains the lines defined by $x_i=-x_j$, $x_k=-x_l$, $x_0=0$ for $\{i,j,k,l\}=\{1,2,3,4\}$.
\end{remark}

The solutions of (HS) that lie in  the hyperplane $\{x_0 = 0\}$ are
precisely the solutions to (SI). However, the multiplicities are different.
If the variety (SI) in $\PP^{n-1}$ is finite, then
 the number of solutions to (\ref{eq:system}) in $\CC^{n}$ equals $a_1a_2 \cdots a_n$ minus the total length of (HS) along~(SI). 
For $n=2$, this observation fully determines the number of solutions 
to (\ref{eq:system}) in terms of $\mathcal{A}$.

\begin{proposition}[$n=2$]\label{prop:n=2}
Assume $a_1 < a_2$.
For generic $(c_1,c_2)\in \CC^2$,
the number of common solutions in $\CC^2$ to the equations
$x_1^{a_1} + x_2^{a_1} = c_1$ and $x_1^{a_2} + x_2^{a_2} = c_2$ equals
$\,a_1 (a_2 - {\rm gcd}(a_1,a_2))\,$ if both
$a_1/{\rm gcd}(a_1,a_2)$ and 
$a_2/{\rm gcd}(a_1,a_2)$ are odd. It equals
the B\'ezout number $a_1a_2$ otherwise. 
\end{proposition}

\begin{proof}
First assume ${\rm gcd}(a_1,a_2)=1$.
The binary forms $x_1^{a_1} + x_2^{a_1}$ and
$x_1^{a_2} + x_2^{a_2}$ are relatively prime,
unless both $a_1 $ and $a_2$ are odd, so
 $x_1+x_2$ divides both forms.
In the former~case, (SI) has no solutions, so the number of 
solutions to (\ref{eq:system}) equals the Bézout number $a_1 a_2$.
If $a_1$ and $a_2$ are odd, then (SI)~$= \{x_1+x_2=0\}$
defines the point $(1:-1)$ on the line $\PP^1$, corresponding to the point $P=(1:-1:0)$ of
the scheme (HS). The multiplicity of (HS) at $P$ can be computed locally in the 
chart $\{x_2\neq0\}$ by setting $x_2=-1$. It is the multiplicity at the point $(0,1)$ of the affine scheme in $\mathbb{C}^2$ defined by the ideal 
$\,I \, = \, \langle \,x_1^{a_1}-x_0^{a_1}-1,\,x_1^{a_2}-x_0^{a_2}-1\rangle\,$. 

Write $m_{P'} =  \langle x_0, x_1-1 \rangle$ for 
the maximal ideal of $P'=(0,1)$ in the local ring $\mathcal{O}_{P'}$ of the curve 
$V(x_1^{a_1}-x_0^{a_1}-1) \subset \CC^2$.  In $\mathcal{O}_{P'}$ we have $x_1-1=\tfrac{x_1^{a_{1}}-1}{u}=\tfrac{x_0^{a_1}}{u}$, where $u$ is a unit.
In fact, $u$ is a certain product of cyclotomic polynomials in $x_1$. Therefore, $x_0$ is a uniformizer,
i.e., $m_{P'} = \langle x_0 \rangle$, and
 $x_1-1$ is contained in $m_{P'}^{a_1}\setminus m_{P'}^{a_1+1}$.
 From this we conclude
 \[
x_1^{a_2}-x_0^{a_2}-1\,\,=\,\,\bigl(((x_1-1)+1)^{a_2}-x_0^{a_2}-1\bigr)\,\,=\,\,\sum_{i=1}^{a_2}\binom{a_2}{i}(x_1-1)^{i}-x_0^{a_2}\,\, \in\,\, m_{P'}^{a_1}\setminus m_{P'}^{a_1+1}.
\]
Hence, $x_1^{a_2}-x_0^{a_2}-1$ vanishes to order $a_1$ at $P'$.
We  conclude that the multiplicity of (HS) in $P$ is $a_1$. Therefore, the system (\ref{eq:system}) has  $a_1(a_2-1)=a_1(a_2-\gcd(a_1,a_2))$ solutions in $\CC^2$. 

Finally, suppose that $a_1 $ and $a_2$ are not relatively prime,
and set $g = {\rm gcd}(a_1,a_2)$. We replace $x_1,x_2$ by $x_1^g, x_2^g$, and we
 apply our previous analysis to the two equations
  \begin{equation}\label{eq:smaller system}\left(x_1^g\right)^{a_1/g} + \left(x_2^g\right)^{a_1/g}\,=\,c_1
  \quad {\rm and} \quad
  \left(x_1^g\right)^{a_2/g} + \left(x_2^g\right)^{a_2/g}\,=\,c_2.\end{equation}
  The system (\ref{eq:smaller system}) has solutions at infinity if and only if $a_1/g$ and $a_2/g$ are both odd. In that case, we have (SI) $= \{x_1^g+x_2^g=0\}$, which defines the $g$ points $(\zeta^i:1)_{i=1,\ldots,g}$ in $\mathbb{P}^1$, where $\zeta$ is a primitive $g$-th root of $-1$. Each of the corresponding points in (HS) has multiplicity $a_1$.
  This can be computed analogously  to what we did for $P$ in the argument above.
\end{proof}

We turn to $n = 3$, and we assume $\gcd(a_1,a_2,a_3)=1$.
 Our problem is now much harder.
It is unknown when (SI) has any solutions in $\PP^2$.
No solutions means that the power sums $\phi_1,\phi_2,\phi_3$ form a regular sequence.
Conca, Krattenthaler and Watanabe 
 \cite[Conjecture 2.10]{CKW} suggest that this holds if and only if
  $a_1a_2a_3\equiv0 \,{\rm mod }\, 6$; we call this the {\em CKW conjecture}.
  They prove the `only if' part in \cite[Lemma 2.8]{CKW}. 
  Another proof for this part is given by the next lemma.
Set $\mathcal{A}_p = \{a_1 \,{\rm mod}\, p,
a_2 \,{\rm mod} \,p,
a_3 \,{\rm mod}\, p \}$ for $p=2,3$.
Thus, $\mathcal{A}_2 \subseteq  \{0,1\} $ and
$\mathcal{A}_3 \subseteq  \{0,1,2\} $.
We assumed $\mathcal{A}_p \not= \{0\}$ for $p=2,3$.
Let $\zeta$ be a primitive cube root of unity.

  \begin{lemma} \label{lem:somesolu}
   The points $(1:-1:0),(1:0:-1),$ and $(0:1:-1)$ are in (SI) if and only if $0\not\in\mathcal{A}_2$, and the points  $(1:\zeta:\zeta^2)$ and $(1:\zeta^2:\zeta)$ are in (SI) if and only if $0\not\in\mathcal{A}_3$. 
  \end{lemma}
  
\begin{proof}
If $n$ is a prime, $\xi$ is a primitive $n$-th root of unity, and
$a$ is a multiple of $n$, then the power sum $x_1^a+x_2^a + \cdots + x_n^a$
does not vanish at $(1,\xi,\ldots,\xi^{n-1})$, but rather
it evaluates to~$n$. We obtain the assertion by specializing to
$n=2$ and $n=3$.
\end{proof}
  
The CKW conjecture states that (SI) has no solutions when
$0 \in \mathcal{A}_2 \cap \mathcal{A}_3$. It is shown in \cite[Theorem 2.11]{CKW} that 
this holds if $\{1,n\}\subset\mathcal{A}$ with $2\leq n\leq7$, or if $\{2,3\}\subset\mathcal{A}$.
The proof rests on the expression of power sums in terms of elementary symmetric polynomials. 

In what follows we present conjectures that imply the CKW conjecture.
We begin with a converse to Lemma \ref{lem:somesolu}.
 Theorems \ref{thm:points at infinity} and \ref{thm:multiplicities} verify all
conjectures for some new cases.

\begin{conjecture}\label{conj.: roots of unity} We have $
\rm{(SI)}\subseteq \{(1\!:\!-1:0),(1:0:\!-1),(0:1:\!-1),(1\!:\!\zeta \!:\!\zeta^2),(1\!:\!\zeta^2\! : \!\zeta)\}$.
\end{conjecture}

This generalizes \cite[Conjecture 2.10]{CKW} since the five possibilities for points on (SI) do
not occur if $0 \in \mathcal{A}_2 \cap \mathcal{A}_3$. We show some new cases of the conjecture
using computational tools.

\begin{theorem}\label{thm:points at infinity}
Conjecture \ref{conj.: roots of unity} holds for all $a_1<a_2<a_3$
with $a_1+a_2+a_3 \leq 300$. 
\end{theorem}

\begin{proof}
Let $P'=(\alpha:\beta:\gamma)$ be a point on (SI), corresponding to $P=(\alpha:\beta:\gamma:0)$~on~(HS). 
After permuting coordinates, we may assume
$\alpha\neq0$. 
Then $P'$ is in the affine chart $\CC^2$ of $\PP^2$ given by
$x = x_2/x_1$ and $y = x_3/x_1$.
The restriction of (SI) to that plane $\CC^2$ is defined by
  \begin{equation}\label{eq: intersection in A2}
    x^{a_1}+y^{a_1}+1 \,
  =\,x^{a_2}+y^{a_2}+1 
\,= \, x^{a_3}+y^{a_3}+1\,\,=\,\,0.\end{equation} 
Conjecture \ref{conj.: roots of unity}  
states that the number of solutions to the system (\ref{eq: intersection in A2})
is  $0$, $2$ or $4$, as follows:
\begin{center}
    \begin{tabular}{c|c|c}
    $\mathcal{A}_2$ and $\mathcal{A}_3$     &  Number of solutions to (\ref{eq: intersection in A2})
     & Possibilities for $P'$ in Conjecture \ref{conj.: roots of unity}
     \\
 \hline
    $0\in\mathcal{A}_2,0\in\mathcal{A}_3$  &  0 & --\\

    $0\notin\mathcal{A}_2,0\in\mathcal{A}_3$     &  2 & $(1:-1:0)$, $(1:0:-1)$\\

    $0\in\mathcal{A}_2,0\not\in\mathcal{A}_3$    &  2 &  $(1:\zeta:\zeta^2)$, $(1:\zeta^2:\zeta)$\\
    $0\not\in\mathcal{A}_2,0\not\in\mathcal{A}_3$   & 4 & $(1\!:\!-1\!:\!0)$, $(1\!:\!0\!:\!-1)$, 
    $(1\!:\!\zeta\!:\!\zeta^2)$, $(1\!:\!\zeta^2\!:\!\zeta)$\\
    \end{tabular}
\end{center}
We verified the counts in the second column for
 all $a_1<a_2<a_3$ with $a_1+a_2+a_3\leq300$.
We did this using the Gr\"obner basis implementation in the computer algebra system
{\tt magma}. The same would be doable with other tools for bivariate equations.
\end{proof}

\begin{conjecture}\label{conj: number of solutions}
For $n=3$ and ${\rm gcd}(a_1,a_2,a_3)=1$, the following holds for the system~(\ref{eq:system}): \\
If $0\in\mathcal{A}_3$, then we have
$\quad\quad\quad\quad\;
\#\mbox{Solutions }= \,
     \begin{cases}
       a_1a_2a_3 &\mbox{ if }\mathcal{A}_2=\{1,0\};\\
       a_1a_2a_3-3a_1a_2 &\mbox{ if }\mathcal{A}_2=\{1\}.\\ 
     \end{cases}
$ \\
If $\mathcal{A}_3=\{1\}$ or $\{2\}$, then we have
$\;
\#\mbox{Solutions }= \,
     \begin{cases}
       a_1a_2a_3-4a_1 &\mbox{ if }\mathcal{A}_2=\{1,0\};\\
       a_1a_2a_3-4a_1-3a_1a_2 &\mbox{ if }\mathcal{A}_2=\{1\}.\\ 
     \end{cases}
$ \\
If $\mathcal{A}_3=\{1,2\}$, then we have
$\,\,\;\;\;\;\;\;\;
\#\mbox{Solutions }= \,
     \begin{cases}
       a_1a_2a_3-2i_{\mathcal{A}} &\mbox{ if }\mathcal{A}_2=\{1,0\};\\
       a_1a_2a_3-2i_{\mathcal{A}}-3a_1a_2 &\mbox{ if }\mathcal{A}_2=\{1\}.\\ 
     \end{cases}
$ \\
Here $i_\mathcal{A}$ is the {\em index of nilpotency} of the zero-divisor $x_0$ in
the homogeneous system (HS).
\end{conjecture}

At present we do not have a simple formula for the number $i_\mathcal{A}$ in all cases.
Computationally, it can be found from the homogeneous ideal $I = \langle f_1,f_2,f_3 \rangle$ that is generated~by $\,f_j=x_1^{a_j}+x_2^{a_j}+x_3^{a_j}-c_jx_0^{a_j}\,$ for $j=1,2,3$.
Using ideal quotients, the definition is as follows:
$$ (I: x_0^{i_\mathcal{A}-1}) \,\,\subsetneq \,\,(I:x_0^{i_\mathcal{A}}) \,\,= \,\,
(I:x_0^{i_\mathcal{A}+1}).
$$
From our computations it seems that $i_\mathcal{A}$ is always either $a_1$ or $a_2$ or $2a_1$.

\begin{theorem}\label{thm:multiplicities}
Conjecture  \ref{conj: number of solutions} holds for all $a_1<a_2<a_3$ with $a_3\leq20$ or $a_1+a_2+a_3\leq40$.
\end{theorem}

\begin{proof}
Our approach is to compute the multiplicity in (HS) for each point that is known 
(by Theorem \ref{thm:points at infinity}) to  lie in (SI).
We conjecture that these multiplicities are as follows:
\begin{itemize}
    \item[](i) If  $\mathcal{A}_2 = \{1\}$, then the point $(1:-1:0:0)$ has multiplicity $a_1a_2$ in  (HS);
    \item[](ii) if $\mathcal{A}_3=\{1,2\}$, then the point $(1:\zeta:\zeta^2:0)$ 
        has multiplicity $i_{\mathcal{A}}$ in (HS);
    \item[](iii) if $|\mathcal{A}_3| = 1$, then  $    2 i_{\mathcal{A}} = 2 a_1$
        and this is the multiplicity of
      $(1:\zeta:\zeta^2:0)$ in (HS).
      \end{itemize}
These claims imply Conjecture \ref{conj: number of solutions},
by our previous analysis. Indeed, if $0 \in \mathcal{A}_2 \cap \mathcal{A}_3$,
then (SI) is empty and the number of solutions to 
(\ref{eq:system}) is the B\'ezout number $a_1 a_2 a_3$. Otherwise,
we need to subtract the multiplicities above, according to the various cases. Here the number in~(i) 
is multiplied by $3$ since the $S_3$-orbit of $(1:-1:0:0)$ has three points,
and the numbers in (ii) and (iii)  are multiplied by $2$ since 
the $S_3$-orbit of $(1:\zeta:\zeta^2:0)$ has two points.

For our computations, we fix $P\in\{(1:-1:0:0),(1:\zeta:\zeta^2:0)\}$,
we focus on the affine chart $\CC^3 = \{x_1 = 1\}$, and 
we consider the ideal $I = \langle f_1,f_2,f_3 \rangle$
in the local ring $\mathcal{O}_{P,\CC^3}$ . 
The quotient $V=\mathcal{O}_{P,\mathbb{A}^3}/I$ is a vector space over $\mathbb{C}$, and its dimension is the multiplicity of (HS) at $P$.  We computed this dimension for all values
of $a_1,a_2,a_3$ in the stated range, and we verified that (i), (ii) and (iii) are satisfied.
This was done using Gr\"obner bases in  {\tt magma}.
\end{proof}

Extending Conjecture \ref{conj: number of solutions}
to $n \geq 4$ seems out of reach at the moment, for two reasons. First of all, the conditions on $\mathcal{A}$ for (SI) to have no solutions are less simple. 
For $n=4$ with ${\rm gcd}(a_1,a_2,a_3,a_4)=1$, Conca, Krattenthaler and Watanabe 
 \cite[Conjecture 2.15]{CKW} state
 three conditions on $\mathcal{A}$ under which (SI) has no solutions.
  They show that all three conditions are necessary. We verified their conjecture using Gr\"obner bases in \texttt{magma} for $a_1+a_2+a_3+a_4 \leq 100$.  
Secondly, in the event that (SI) does have solutions, it is not at all obvious what these should be. In general, they are not given only by points whose coordinates are roots of unity, 
as was the case for $n=3$. This happens already for $n=4$ as the following example shows:

\begin{example}Set $n=4$ and $\mathcal{A} = \{2,4,9,10\}.$ The system (\ref{eq:system}) has $576$ solutions which is $144$ less than the Bézout number $720$. This is explained
by the scheme (SI) in $\PP^3$ which is defined by
 the ideal $\langle x_1^2+x_2^2+x_3^2+x_4^2,x_1^4+x_2^4+x_3^4+x_4^4,x_1^9+x_2^9+x_3^9+x_4^9,x_1^{10}+x_2^{10}+x_3^{10}+x_4^{10}\rangle$. 
 The minimal polynomial of each of the coordinates of the points in (SI) has degree $36$.
\end{example}

\section{Images}
\label{sec3}

We now study the images
of the power sum maps $\phi_{\mathcal{A},\CC}$,
$\phi_{\mathcal{A},\RR}$ and
$\phi_{\mathcal{A},\geq 0}$.
The recovery problem~(\ref{eq:system}) has a solution if and only if
the measurement vector $c$ lies in that image.
We know from Chevalley's Theorem \cite[Theorem 4.19]{MS}
that ${\rm im}(\phi_{\mathcal{A},\CC})$ is a constructible subset of $\CC^m$.
Over the real numbers, the
Tarski-Seidenberg Theorem \cite[Theorem 4.17]{MS}
tells us that ${\rm im}(\phi_{\mathcal{A},\RR})$ is a semialgebraic subset of $\RR^m$
and ${\rm im}(\phi_{\mathcal{A},\geq0})$ is a semialgebraic subset of~$\RR_{\geq 0}^m$.
It follows from Proposition \ref{lem: complete intersection} that, for each of these images, the dimension equals~${\rm min}(n,m)$.

We begin with the question whether the  images are closed.
We use the classical topology  on $\RR^m$ or  $\CC^m$.
This makes sense not just over $\RR$, but also over $\CC$, since the 
Zariski closure of any complex polynomial map  coincides with its
classical closure \cite[Corollary~4.20]{MS}.

\begin{proposition}
The constructible set $\,{\rm im}(\phi_{\mathcal{A},\CC})$ is generally not closed in $\CC^m$. The semialgebraic set $\,{\rm im}(\phi_{\mathcal{A},\RR})$ is closed in $\,\RR^m$ when $\,0\in\mathcal{A}_2$, but it is generally not closed
otherwise.
Finally, the semialgebraic set $\,{\rm im}(\phi_{\mathcal{A},\geq 0})$ is always closed in $\RR^m_{\geq 0}$.
 \end{proposition}

\begin{proof}
Let $m=n=2$, $a_1 = 1$ and  $a_2 \geq 3 $ odd. Fix any $c=(c_1,c_2) \in \CC^2$, with $c_1 = 0$ and~$c_2 \not= 0$. Then (\ref{eq:system}) has no complex solution because
$\phi_1(x)$ divides $\phi_2(x)$. Thus, $c$ does not belong to the image of $\phi_{\mathcal{A},\CC}$ or $\phi_{\mathcal{A},\RR}$. However, $c$ is in the closure of 
$\,{\rm im}(\phi_{\mathcal{A},\CC})$ because that closure is $\CC^2$ by 
Proposition \ref{lem: complete intersection} (i). The same counterexample works over the real numbers.
Let us now set $c_1 = \epsilon$ for $\epsilon > 0$ very small 
and solve $\phi_1(x) = \epsilon$
by setting $x_2 = \epsilon - x_1$. This substitution in
$\phi_2(x) = c_2$ gives a polynomial equation in one variable $x_1$ of odd degree $a_2$.
Such an equation always has a real solution $x_1(\epsilon)$.
The image of the point $(x_1(\epsilon),  \epsilon - x_1(\epsilon))$
converges to $c$ in $\RR^2$ as $\epsilon \rightarrow 0$ from which we conclude that $c$ lies in the closure of ${\rm im}(\phi_{\mathcal{A},\RR})$. 

We are left with the cases where the image is closed. First suppose that $a_i \in \mathcal{A}$ is~even. Let $c \in \RR^m$ be  in the closure of $\,{\rm im}(\phi_{\mathcal{A},\RR})$.
There exists a sequence $\{x^{(\ell)} \}_{\ell \geq 0}$ of points in 
$\RR^n$ such that $\phi_{\mathcal{A},\RR}(x^{(\ell)})$ converges to $c$ as $\ell \rightarrow \infty$.
Since $a_i$ is even, we have
$\phi_i(x^{(\ell)})= || x^{(\ell)} ||_{a_i}^{a_i} $, which converges to $c_i \geq 0$ as $\ell \rightarrow \infty$.
Hence, the sequence $\{x^{(\ell)} \}_{\ell \geq 0}$  is bounded in the norm
$|| \cdot ||_{a_i}$. There is a subsequence that converges
to some point $x^{(\infty)}$ in~$\RR^n$. Since the power sum map $\phi_{\mathcal{A},\RR}$ is
continuous, the image of $x^{(\infty)} $ is equal to $c$.
Therefore,  $c \in {\rm im}(\phi_{\mathcal{A},\RR})$, and we conclude that the real image is closed. Finally, take $\mathcal{A}$ arbitrary and consider the nonnegative power map $\phi_{\mathcal{A},\geq 0}$. Let $c \in \RR^m_{\geq 0}$ be in the closure of $\,{\rm im}(\phi_{\mathcal{A},\geq 0})$ and let $\{x^{(\ell)} \}_{\ell \geq 0}$ be a sequence of nonnegative points
whose images $\phi_{\mathcal{A},\geq 0}(x^{(\ell)})$ converge to $c$ as $\ell \rightarrow \infty$.
Now, for any index $i$, the norm $ || x^{(\ell)} ||_{a_i} = \phi_i( x^{(\ell)})^{1/a_i}$ is bounded,
so there exists a convergent subsequence of $\{x^{(\ell)}\}_{\ell\geq0}$. Let $x^{(\infty)} \in\RR^n_{\geq 0}$ be its limit. Again, by the continuity of the power map, we have $\phi_{\mathcal{A},\geq 0}(x^{(\infty)})=c$. From this we conclude that $\,{\rm im}(\phi_{\mathcal{A},\geq 0})$ is closed.
\end{proof}

\begin{example}
Set $n=m=2$ and $\mathcal{A} = \{1,3\}$.
The image of $\phi_{\mathcal{A},\RR}$ is the non-closed set
$$ \{ (0,0) \} \,\,\cup \,\, \bigl\{ \,c \in\RR^2\,:\,
(c_1 < 0 \,\,{\rm and} \,\, c_1^3 \geq 4c_2 ) \,\,\,{\rm or} \,\,\,
(c_1 > 0 \,\,{\rm and} \,\,  c_1^3 \leq 4c_2 ) \, \bigr\}. $$
On the other hand, the image of  the map restricted to the nonnegative orthant is closed:
\begin{equation}
\label{eq:13posimage}
\,{\rm im}(\phi_{\mathcal{A},\geq 0}) \,\, = \,\,
\bigl\{ \,c \in\RR^2_{\geq 0} \,:\, c_2 \leq c_1^3 \leq 4 c_2\, \bigr\}. 
\end{equation} In Section \ref{sec: recovery from norms} we generalize this description of the image of  $\phi_{\mathcal{A},\geq0}$ to other power sum maps.
\end{example}

We next examine our images through the lens of algebraic geometry.
Let $c_1,\ldots,c_m$ be variables with 
${\rm degree}(c_i) = a_i$. These are coordinates on the 
weighted projective space $\WP^{m-1}$ with weights
given by $\mathcal{A}$. We regard $\phi=\phi_{\mathcal{A},\CC}$ as a
 rational map from $\PP^{n-1}$ to $\WP^{m-1}$.
 The following features of the image will be characterized in
Theorem \ref{thm:weighted}: (i) For $m=n+1$, the closure of 
the image ${\rm im}(\phi)$ is an irreducible
hypersurface in $ \WP^{m-1}$. We give a formula for its degree, which is the weighted degree of its defining polynomial in 
the unknowns $c_1,\ldots,c_m$. (ii) For $m \leq n$, we describe the  \textsl{positive branch locus} of the map $\phi$. This is a hypersurface in $\CC^m$. By reasoning as in the proof of \cite[Theorem 3.13]{KKK}, this 
hypersurface represents the algebraic boundary of the image of $\phi_{\mathcal{A},\geq0}$.

To study the branch locus of $\phi$, we start with the ramification locus $\mathcal{R}$.
This consists of points in $\CC^n$ where $\phi$ is not smooth \cite[\wasyparagraph 2.2, Proposition 8]{BLR90}. Set $\phi_j=\sum_{i=1}^nx_i^{a_j}$  and
 $\mu=\min\{n,m\}$.
Let $I\subset\CC[x_1,\ldots,x_n]$ be the ideal generated by the
$\mu \times \mu$ minors of the Jacobian matrix $\mathcal{J}$ as in (\ref{eq:jacobian}). 
Then $\mathcal{R} = V(I)$ is the set of points where $\mathcal{J}$ has rank less than~$ \mu$.
  Each maximal minor of $\mathcal{J}$, up to multiplication with a positive integer, 
  has the form 
\begin{equation}\label{eq:maximal minors}
\left(x_{i_1} x_{i_2}\cdots x_{i_{\mu}}\right)^{a_{i_1}-1} \cdot \!\!
\prod_{1\leq j<k\leq \mu} \!\!\! \!(x_{i_j}-x_{i_k})\,\cdot\,  S(x),\end{equation} for some 
$1 {\leq} i_1{<}\cdots {<} i_{\mu} {\leq}                  n$ and $i\in\{1,\ldots,m{-}\mu{+}1\}$.
The last factor  is a Schur~polynomial. 

In the square case $m=n$, the variety $\mathcal{R}$ is a reducible hypersurface in $\CC^n$, given by the vanishing of one polynomial (\ref{eq:maximal minors}). Write $g=\mbox{gcd}(a_1-1,\ldots,a_m-1)$.
 By \cite[Theorem 3.1]{DZ09}, the Schur polynomial $S$ is either constant, which happens  when $(a_i-1)/g=i-1$ for $1\leq i\leq n$, or it is irreducible. Let $\mathcal{R}'$ be the 
 closure in $\CC^n$ of  $\mathcal{R}\setminus \left(\cup_{i\neq j}V(x_i-x_j)\cup V(x_i)\right)$. 
 Thus, $\mathcal{R}'$ is the non-trivial component in the ramification locus.
 Our discussion implies the following:

\begin{proposition}
Assume $m=n$. The ramification variety $\mathcal{R}'$ is either empty, in which case we have $(a_i-1)/g=i-1$ for $1\leq i\leq n$, or it is an irreducible hypersurface of degree  $$\sum_{i=1}^n(a_i-1)-\binom{n}{2}-n(a_1-1).$$
\end{proposition}

\begin{example}
For $\mathcal{A}=\{3,6,7\}$ and $n=3$, the ideal $I$ is principal. Its generator factors as
  $$x_1^2 x_2^2 x_3^2 (x_1-x_2)(x_1-x_3)(x_2-x_3)(x_1^2x_2^2+x_1^2x_2x_3
+x_1x_2^2x_3+x_1^2x_3^2+x_1x_2x_3^2+x_2^2x_3^2).$$
The variety $\mathcal{R}'$ is the quartic surface
in $\CC^3$ defined by the Schur polynomial in the last factor.
\end{example}

In the overdetermined regime ($m>n$), we remove only
the hyperplanes $\{x_i = x_j\}$.
 Let $\mathcal{R}''$ denote the closure 
of $\mathcal{R}\setminus \left(\cup_{i\neq j}V(x_i-x_j)\right)$ in $\CC^n$.
The ramification part $\mathcal{R}''$ was studied by
Fr\"oberg and Shapiro \cite{FS16}, but they reached only partial results. Notably,
it remains an open problem to find the dimension of $\mathcal{R}''$.
Assuming $a_1=g=1$,  the first interesting case is $n=3$, $m=5$.
This was studied in \cite{FS16}.
Proving the dimension of $\mathcal{R}''$ to be the expected one is equivalent to showing that three complete homogeneous polynomials form a regular sequence. This brings us
back to  Conca, Krattenthaler and Watanabe \cite[Conjecture~2.17]{CKW}.

\smallskip

Now assume $m\leq n$. Then $\mathcal{R}$ contains all linear spaces
defined by $n-m+1$ independent equations of the form
$x_i = x_j$ or $x_k = 0$. We call these {\em positive ramification components}.
This name is justified as follows. The Schur polynomial
 $S$ in (\ref{eq:maximal minors}) has
 positive coefficients, and therefore cannot vanish at nonzero points
in the nonnegative orthant $\RR^n_{\geq 0}$.
Hence, only these components contribute to the ramification 
locus of the positive map $\phi_{{\mathcal A}, \geq 0}$.
A {\em positive~branch hypersurface} is any irreducible hypersurface
in weighted projective space in $\WP^{m-1}$  that is the
closure of the image of a positive ramification component under the
power sum map~$\phi$.

\begin{theorem} \label{thm:weighted}
The following hypersurfaces in $\WP^{m-1}$ are relevant for the image of our map.
\begin{itemize}
    \item[(i)] If $m=n+1$, then the image of $\phi$  
is~an irreducible hypersurface in $\WP^{m-1}$ whose weighted degree is at most
$(a_1 a_2 \cdots a_m)/(m-1)!$.  If this ratio is an integer, then 
this bound can be attained. Specifically,
if $m=3$ and $a_1 a_2 a_3$ is even, then it can be  attained. 
\item[(ii)] If $m \leq n$, then the weighted degree of any positive branch hypersurface of 
$\phi$ is at most the B\'ezout number $a_1 a_2 \cdots a_m$.
\end{itemize}
\end{theorem}

\begin{proof}
Let $m \leq n$. The restriction of $\phi$ to any
positive ramification component is a rational map from
$ \PP^{m-2} $ to $\WP^{m-1}$. After renaming the $x_i$ if needed, we can write its 
coordinates as 
 \begin{equation}\label{eq:restricted map}\sum_{i=1}^{m-1} \tau_i x_i^{a_j}
 \qquad \mbox{ for }j=1,\ldots,m,\end{equation}
  where $\tau_1,\ldots,\tau_{m-1}$ are positive integers. 
  Let $H_\tau$ denote the image of this map in $\WP^{m-1}$.
This also covers the case $m=n+1$ in (i) since $\phi$ has coordinates as in (\ref{eq:restricted map}) with $\tau_1 = \cdots =\tau_{m-1} = 1$. 
Hence, all hypersurfaces in (i) and (ii) have the form $H_\tau$.
  Our aim is to compute their degrees.
    
Fix positive integers $\tau_1,\ldots,\tau_{m-1}$ and set $z_j = c_j^{1/a_j}$.
Consider the projective space  $\PP^{2m-2}$ with coordinates
$x_1,\ldots,x_{m-1},z_1,\ldots,z_m$. Let $Z$ denote the variety in $\PP^{2m-2}$
defined by the homogeneous polynomials 
$\sum_{i=1}^{m-1} \tau_i x_i^{a_j} -  z_j^{a_j}$, for $j=1,\ldots,m$.
By the same reasoning as in Proposition \ref{lem: complete intersection}, this variety is irreducible and it is a complete intersection of degree $a_1 a_2 \cdots a_m$. 

We consider the image of $Z$ under the coordinate projection
\begin{equation}
\label{eq:map2m-2}
 \pi\colon\PP^{2m-2} \dashrightarrow \PP^{m-1},\,
(x_1:\cdots:x_{m-1}:z_1:\cdots:z_m) \mapsto (z_1:\cdots:z_m). 
\end{equation}
The closure of $\pi(Z)$ is essentially the hypersurface $H_\tau$ we care about,
but it lives in~$\PP^{m-1}$. Its degree in $\PP^{m-1}$
with coordinates $(z_1:\cdots:z_m)$ coincides with the degree
of $H_{\tau}$ in $\WP^{m-1}$ with coordinates $(c_1:\cdots:c_m)$.
Indeed, these two hypersurfaces have the same defining polynomial, up to the substitution
$c_j = z_j^{a_j}$. 
The Refined B\'{e}zout Theorem \cite[12.3]{Ful} implies  \begin{equation}\label{eq:degree of image}\deg(\pi(Z))\,\, \leq \,\,\frac{\deg(Z)}{\deg(\pi|_Z)}
\,\,= \,\,\frac{a_1 a_2 \cdots a_m}{\deg(\pi|_Z)},\end{equation} where equality holds if $\pi$ has no base locus.
This immediately proves (ii). 

We proceed with proving (i). The degree of $\pi|_Z$ is the size of its generic fiber.
This equals the size of the generic fiber of the map given by (\ref{eq:map2m-2}). Conjecture \ref{conj:injectivity2} states that size of the generic fiber of $\pi$ is the size of the stabilizer of $\tau=(\tau_1,\ldots,\tau_{m-1})$ in the symmetric group~$S_{m-1}$. In particular, it would follow that the generic fiber is a single point if and only if the $\tau_i$ are all distinct, and it consists of $(m-1)!$ points if and only if the $\tau_i$ are identical. 
We do not know yet whether this conjecture holds. 
But, in any case, the number $|{\rm Stab}(\tau)|$ furnishes a lower bound for the size of a generic fiber and thus for $\deg(\pi)$.

Since the hypersurface ${\rm im}(\phi)$ in (i) equals $H_{\tau}$ for $\tau_1=\cdots=\tau_{m-1}=1$, we conclude from (\ref{eq:degree of image}) that its weighted degree is at most $(a_1 a_2 \cdots a_m)/(m-1)!$. Equality can only hold when the base locus of $\pi$ on the variety
$V(\phi_j(x) -  z_j^{a_j}\,:\,j=1,\ldots,m)$ is empty. This happens precisely when the system at infinity (SI) is the empty set.
A necessary~condition for this to happen is that $(m-1)!$
divides the B\'ezout number $a_1 a_2 \cdots a_m$.
One checks that this
is also sufficient when $m=3$: we saw
in Proposition \ref{prop:n=2},
that (SI) is empty when $a_1 a_2$ is even.
\end{proof}

\begin{example}[$m=3$]
Suppose $ {\rm gcd}(a_1 ,a_2,a_3)=1$  and
$B = a_1a_2a_3$ is even. If $n=2$, then the
image of $\phi$ is a curve of expected degree $B/2$ in $\WP^2$.
If  $n\geq 3$, then every positive branch curve has expected degree 
$B/2$ or $B$.
For instance, if $n=4$, then the ramification component
$\{x_1= 0, x_3 = x_4\}$ should give a 
branch curve of degree $ B$,
while $\{x_1=x_2,x_3 =x_4\}$ should give a branch curve of degree~$B/2$.
We shall see pictures of such curves in the next section.
\end{example}

\begin{example}[$m=4$]
If $n=3$ and $6$ divides $B =a_1a_2 a_3 a_4$, then we expect the image of $\phi$ to have weighted degree $B/6$. This would follow from the conjectures in
  Sections \ref{sec:fibers} and \ref{sec: square systems}. 
The positive branch surfaces for $n \geq 4$ should have degrees
$B/6$, $B/2$ or $B$.
If $6$ does not divide $B$, then the weighted degrees of the image and branch
surfaces in $\WP^3$ are determined by the base loci. This
 takes us back to Conjecture~\ref{conj: number of solutions}.
To be very explicit, let $\mathcal{A} = \{2,5,7,8\}$  as in Example \ref{ex:regimes}.
Here $B/6 = 560/6 = 93.333...$. The image 
of our map $\phi: \PP^2 \dashrightarrow \WP^3$
is defined by
a homogeneous polynomial of weighted degree $90$ with $304$ terms, namely
\begin{small}  $$
9 c_1^{45}-1050 c_1^{41} c_4-3724 c_1^{40} c_2^2+22400 c_1^{39} c_2 c_3
-31000 c_1^{38} c_3^2+ \cdots
-1966899200 c_1 c_4^{11}+1258815488 c_2^2 c_4^{10}. $$ \end{small}
By contrast, consider $\mathcal{A} = \{2,5,7,9\}$. Now,
$B/6 = 105$ is an integer, and this equals the weighted degree of the image surface. Its 
defining polynomial has $388$ terms, and it looks~like 
\begin{small} $$
59049 c_1^{35} c_2^7-459270 c_1^{34} c_2^6 c_3-59049 c_1^{35} c_3^5+255150 c_1^{33} c_2^6 c_4+ \cdots  +6350400 c_2 c_3^4 c_4^8-324000 c_3^6 c_4^7.
$$ \end{small}
\end{example}

\section{Recovery from $p$-norms}\label{sec: recovery from norms}

Focusing on the positive region, we now
 investigate the properties of the map $\phi_{\mathcal{A},\geq0}$. 
 The key fact to be used throughout is that the power sum of degree $p$
 represents the $p$-norm:
 \begin{equation}\label{eq:equivalence power sums and norms}
 ||x||_p \,\, = \,\,\biggl(\,\sum_{i=1}^nx_i^{p}\,\biggr)^{\! 1/p}\;\;\;\;\;\mbox{ for all }\, x=(x_1,\ldots,x_n)\in\RR^n_{\geq0}. \end{equation} 
 Hence, our recovery problem for nonnegative vectors $x \in\RR^n_{\geq0}$ is equivalent
 to recovery of $x$ from values of the $p$-norms $||\cdot||_{p}$, where $p$
 runs over a prespecified set $\mathcal{A}$ of positive integers.
We are interested in existence and uniqueness of vectors
with given $p$-norms for $p \in \mathcal{A}$.

Let us begin with the basic identifiability question: How many different $p$-norms are needed to reconstruct a vector in $\RR^n_{\geq0}$ from their values up to permuting the $n$ coordinates?
Conjecture \ref{conj.:injectivity over CC} together with (\ref{eq:equivalence power sums and norms}) would imply that $n+1$ different norms suffice.
  On the other hand, it follows from Proposition \ref{lem: complete intersection} that at least $n$ different $p$-norms are necessary. But are these 
  $n$ measurements already sufficient? We start by showing that this is
indeed the case.

\begin{proposition}\label{prop: unique recovery of norms}
For $m=n$, recovery from $p$-norms is always unique.
Given any set $\mathcal{A}$ of $n$ positive integers,
the map $\phi_{\mathcal{A},\geq0} : \RR^n_{\geq 0} \rightarrow \RR^n_{\geq 0}$ is injective up to permuting coordinates.
\end{proposition}

\begin{proof}
Write $\phi=\phi_{\mathcal{A},\geq0}$. We proceed by induction on $n$. For $n=1$, the map $\phi$ is obviously injective as it is strictly increasing. Thus, we have unique recovery for $n=1$. Let us now prove the statement for arbitrary $n$. Our argument is based on the calculus fact that a differentiable function from a real interval to $\RR$  is injective if its derivative has constant sign. 

Consider the cone of decreasing vectors,
$C = \{x \in \RR^n: x_1 > x_2  > \cdots > x_n \geq 0 \}$. Let $X_1,X_2$ be two arbitrary distinct points from this cone. Our claim states that they map to two different points under the map $\phi$, i.e., that $\phi$ is injective on $C$. Let $L$ be the line segment from $X_1$ to $X_2$ and consider the restriction $\phi|_L:\RR\rightarrow\RR^n$ of $\phi$ to $L$ that is now a function in one variable. Its derivative is given by the product of the Jacobian matrix of $\phi$, which we denote by $\mathcal{J}_\phi$, evaluated at $L$, and the vector $(X_2-X_1)$. First notice that if $X_{1,n}=X_{2,n}=0$, then we are in the case where the induction hypothesis applies. Let us now w.l.o.g.~assume $X_{2,n}>0$. Then $\mathcal{J}_\phi$ is an $n\times n$ matrix whose determinant is of the form (\ref{eq:maximal minors}). 

The Schur polynomial $S$ does not vanish on $L\setminus\{X_1\}$, and neither do the linear factors. 
Hence, the coordinates of the vector
$\mathcal{J}_\phi\cdot(X_2-X_1)$ do not vanish at any point on $L\setminus\{X_1\}$.
Each coordinate is a function of constant sign on the whole segment $L$. This shows that $\phi$ is injective on the line $L$. As $X_1$ and $X_2$ were chosen to be arbitrary points, we conclude that $\phi$ is injective on
the whole cone $C$. For a much more
general version of this argument, we refer to the equivalence of conditions $(inj)$ and
$(jac)$ in \cite[Theorem 1.4]{MFRCSD}.
\end{proof}

Our next goal is to characterize the semialgebraic set $\,{\rm im}(\phi_{\mathcal{A},\geq0})$
inside the nonnegative orthant $\RR^m_{\geq 0}$.
Starting with  $m=2$, we first present a generalization
of the formula (\ref{eq:13posimage}).

\begin{proposition}\label{prop: image positive map} Set $ m = 2 \leq n$ and $\mathcal{A} = \{ a_1 < a_2 \}$.
Then  the nonnegative image equals
\begin{equation}
\label{eq:posimage}
\,{\rm im}(\phi_{\mathcal{A},\geq 0}) \,\, = \,\,
\bigl\{ \,c \in\RR^2_{\geq 0} \,\,:\,\, c_2^{a_1}\, \leq \,c_1^{a_2} \,\leq \,n^{a_2-a_1} c_2^{a_1}\, \bigr\}. 
\end{equation}
\end{proposition}

\begin{proof}
At any point $x \in \RR^n_{\geq 0}$, our map evaluates the norms
$ ||x||_{a_j} = \phi_j(x)^{1/a_j}$ for $i=1,2$.
The first norm  is  larger than or equal to the second one: $||x||_{a_1} \geq ||x||_{a_2}$.
They agree at coordinate points. Their ratio is
maximal at $e=(1,1,\ldots,1)$. This gives the inequalities
$$ 1 \,\,\leq \,\,\frac{||x||_{a_1}}{||x||_{a_2}} \,\,\leq  \,\,
\frac{||e||_{a_1}}{||e||_{a_2}} \,=\, \frac{n^{1/a_1}}{n^{1/a_2}} . $$
All values in this range are obtained by some point $x \in \RR^{n}_{\geq 0}$.
We now raise both sides to the power $a_1 a_2$ and thereafter we clear
denominators. This gives the inequalities in (\ref{eq:posimage}).
\end{proof}

The proof of Proposition \ref{prop: image positive map} suggests that the study
of the nonnegative image  $\,{\rm im}(\phi_{\mathcal{A},\geq 0}) \,$
can be simplified by replacing the power sum map by the
normalized map into the simplex
\begin{equation}
\label{eq:psimap}
 \psi_{\mathcal{A}} \,:\,\RR^n_{\geq 0} \,\dashrightarrow \,\Delta_{m-1} \,:\,\,
x \,\mapsto\, \frac{1}{\sum_{j=1}^m || x ||_{a_j}} \cdot 
\bigl(\, ||x||_{a_1},||x||_{a_2}, \ldots,||x||_{a_m}  \bigr). 
\end{equation}
Here, $\Delta_{m-1} = \{u \in \RR^m_{\geq 0} \,: \, u_1 + u_2 + \cdots + u_m = 1 \}$
is the standard probability simplex. If we know the image of this map then
that of the power sum map can be recovered as follows:
\begin{equation}\label{eq:image norm map and positive image} \,{\rm im}(\phi_{\mathcal{A},\geq 0}) \,\, = \,\,
\bigl\{ \,c \in \RR_{\geq0}^m \,\,: \,\,\frac{1}{\sum_{j=1}^m c_j^{1/a_j}} 
\bigl(\,c_1^{1/a_1}, \,c_2^{1/a_2} , \ldots, \,c_m^{1/a_m}\, \bigr) \, \in \,{\rm im}(\psi_{\mathcal{A}} )\, \bigr\}. \end{equation}

We next consider the case $m=3$. For every $n \geq 3$, the image
is a nonconvex region in the triangle $\Delta_2$. These regions get larger
as $n$ increases. We illustrate this for an example.

\begin{figure}[h]
\begin{center}
\includegraphics[scale=0.30]{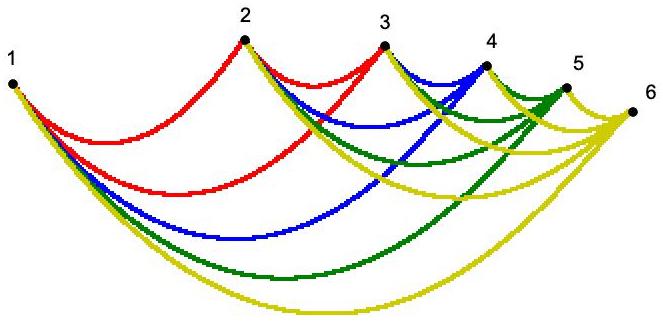} \qquad 
\includegraphics[scale=0.30]{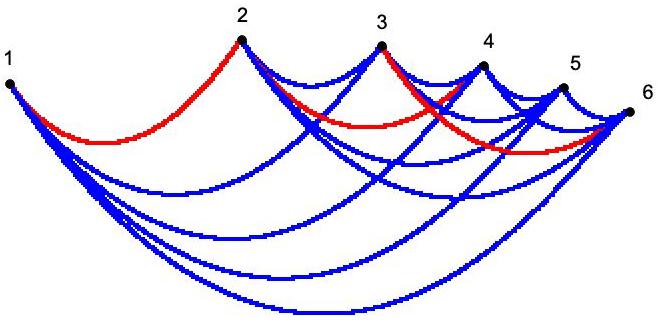}
\vspace{-0.1in}
\end{center}
\caption{The image of the norm map $\psi_\mathcal{A}$ 
for $n=6,m=3$ is a curvy hexagon in a triangle. The color coding on
the left shows the progression of images for $n=3,4,5,6$.
The color coding on the right
shows the algebraic degrees $12$ (red) and $24$ (blue) of
the curvy segments.	}
\label{fig:waves}	
\end{figure}

\begin{example} \label{ex:waves}
Set $\mathcal{A}=\{2,3,4\}$. For $n\geq 3$, the image of 
the norm map $\psi_\mathcal{A}$ into the triangle $\Delta_2$
is an $n$-gon with curvy boundary edges that lies inside the subtriangle $\{c_1>c_2>c_3\}$.
The edges and diagonals of this $n$-gon
are the following $\binom{n}{2}$ curvy segments  for $1 \leq i < j \leq n$:
$$ B_{ij} =  \psi_\mathcal{A} \bigl( \{ x \in \RR^n_{\geq 0} \,:\,
x_1 = \cdots = x_i
\geq x_{i+1} = \cdots = x_{j} \geq x_{j+1} = \cdots = x_n = 0 \}\bigr). $$
The Zariski closure of $B_{ij}$ is an irreducible curve. 
There are $\lfloor n^2/4 \rfloor$ distinct branch curves in total.
For $n=3$, there are two distinct branch curves:
one curve of degree $12$, given by the segment $B_{12}$,
and one of degree $24$, given by the two segments $B_{13}$ and $B_{23}$.
For $n=6$,
Figure \ref{fig:waves} shows the 
curvy hexagon ${\rm im}(\psi_\mathcal{A})$.
Its  $15$ curvy segments form nine distinct branch curves,
six of degree $24$ and three of degree $12$. The latter 
are given by $B_{12}, B_{24},B_{36}$. 
The curvy segment $B_{12}$ is red in both pictures.  For $n=2$, 
we have $B_{12} = {\rm im}(\psi_\mathcal{A})$.
For $n \geq 3$, the curvy segment $B_{12}$ is one of the $n$ boundary edges of 
${\rm im}(\psi_\mathcal{A})$. The Zariski closure of the 
curvy segment $B_{12}$ is the branch curve
$\{c_1^{12}-4c_1^6c_2^6-4c_2^{12}+12c_1^2c_2^6c_3^4-3c_1^4c_2^8-2c_3^{12}=0\}$.
\end{example}

We now state a theorem which generalizes our observations in Example  \ref{ex:waves}
to $m \geq 4$. We fix $n \geq m$ and
$\mathcal{A} = \{a_1 < \cdots < a_m\}$ as before.
For $1 \leq \ell \leq m$ and any ordered set $\nu=(\nu_1,\dots,\nu_l) \in \binom{[n]}{\ell}$,
let $R_\nu$ denote the set of vectors
$ x \in \RR^n_{\geq 0} $ that satisfy
$x_i = x_{i+1}$ if $i < \nu_1$ or $\nu_r \leq i < \nu_{r+1}$ for some~$r$, and $x_i=0$ for all $i>\nu_\ell$,
and $x_i \geq x_{i+1}$ otherwise.
Its image  $B_\nu = \psi_{\mathcal{A}}(R_\nu)$ is a 
semialgebraic subset of  dimension $\ell-1$ in $\Delta_{m-1}$.
Proposition \ref{prop: image positive map} tells us that
$B_\nu$ is a curvy simplex with vertices
$B_{\nu_1},\ldots,B_{\nu_\ell}$.
We define the {\em type} of $\nu$ to be the multiset
$\{\nu_1, \nu_2-\nu_1, \nu_3-\nu_2,\ldots,\nu_\ell- \nu_{\ell-1}\}$.
We can view $\tau = {\rm type}(\nu)$ as a partition 
with precisely $\ell$ parts of
an integer between $\ell$ and $n$.
Let $T_{n,\ell}$ denote the set of such partitions $\tau$.
We use the notation ${\rm Stab}(\tau)$ from
Conjecture \ref{conj:injectivity2}. 
In analogy to the proof of Theorem \ref{thm:weighted}, we denote by $H_{\tau}$ the image in the simplex $\Delta_{m-1}$ of a positive ramification component of type~$\tau$.

\begin{theorem} \label{thm:waves} Assume $m\leq n$. 
The norm map $\psi_{\mathcal{A}}$ in
(\ref{eq:psimap}) has the following properties:
\begin{itemize}
    \item[](i) The image of $\psi_{\mathcal{A}}$ in $\Delta_{m-1}$ is the union of 
the curvy $(m - 1)$-simplices $B_\nu$ where $\nu \in \binom{[n]}{m}$.
The curvy facets of these curvy simplices are 
$B_\mu$ where $\mu \in \binom{[n]}{m-1}$.
Some of these curvy $(m-2)$-simplices form the boundary  of the semialgebraic set $\,{\rm im}(\psi_{\mathcal{A}})$.
\item[](ii) Two curvy $( m-2)$-simplices $B_\mu$ and $B_{\mu'}$ have the
 same Zariski closure if $\,{\rm type}(\mu) = {\rm type}(\mu')$.
Thus, the irreducible branch hypersurfaces $H_\tau$ are indexed by $\,\tau \in T_{n,m-1}$.
\end{itemize}
 \end{theorem}

\begin{proof}
 For $\ell\in\{1,\ldots,m\}$ and $\nu\in\binom{[n]}{\ell}$, the set $R_{\nu}$ is a convex polyhedral cone, spanned by linearly independent vectors in
 a linear subspace of dimension $\ell\leq m$ in $\RR^n$. 
  By Proposition~\ref{prop: unique recovery of norms}, the map $\phi_{\mathcal{A},\geq 0}$ is injective on~$R_{\nu}$. Therefore, by the transformation in (\ref{eq:image norm map and positive image}),
the map  $\psi_{\mathcal{A}}$ is injective on~$R_{\nu}$ up to scaling. This means that
 the image $B_{\nu} = \psi_\mathcal{A}(R_{\nu})$ is a
  curvy simplex of dimension $\ell - 1 $ inside the probability simplex $\Delta_{m-1}$. 
  We also conclude that the boundary of  ${\rm im}(\psi_\mathcal{A})$ equals
  the union of the images $B_{\nu}=\psi_{\mathcal{A}}(R_\nu)$, where $\nu$
   runs over a certain subset of $\binom{[n]}{m-1}$. These
   specify the algebraic boundary of ${\rm im }(\phi_{\mathcal{A},\geq 0})$. This proves~(i).
 
To see that part (ii) holds, we write  the restriction of $\phi_{\mathcal{A},\geq0}$ to
the cone $R_\mu$ as a polynomial function in only $\ell$ distinct variables $x_i$. The
$j$-th coordinate of that restriction has the form $\sum_{i=1}^\ell \tau_i\, x_i^{a_j}$,
where $\tau = {\rm type}(\mu)$. Different cones $R_\mu$ of the same
type $\tau$ are distinguished only by the orderings of the parameters
$x_1,x_2,\ldots,x_\ell$. However, they have the same linear span in $\RR^n$.
Hence, after we drop the distinguishing inequalities $x_i > x_j$,
the maps are the same. In particular,
their images $B_\mu$ have the same Zariski closures 
$H_\tau$ in the simplex
$\Delta_{m-1}$.
\end{proof}

Example \ref{ex:waves} illustrates Theorem \ref{thm:waves} for $m{=}3$,
where $|T_{n,2}| = \lfloor n^2/4 \rfloor $ and
$|{\rm Stab}(\tau)| \in \{1, 2\}$.
We found it more challenging to understand the geometry of our image in higher dimensions.

\begin{example}[$n=8,m=4$] \label{ex:morewaves}
The image of $\psi_{\mathcal{A}}$ in the tetrahedron is a curvy $3$-polytope.
It is partitioned by $56 = \binom{8}{3}$ curvy triangles $B_\nu$.
Their types $\tau$ identify $16$ clusters: two singletons, ten triples, and
four of size six. These determine $16 = |T_{8,3}|$ branch surfaces~$H_\tau$.
\end{example}

Based on computational experiments,
we believe that,
for all pairs $m \leq n$ and all
exponents $\mathcal{A}$, the image of $\psi_{\mathcal{A}}$ has the 
combinatorial structure of the {\em cyclic polytope} of dimension
$m-1$ with $n$ vertices. In particular, the boundary is formed
by the curvy $(m-2)$-simplices $B_\mu$ where $\mu$
runs over all sequences that satisfy {\em Gale's Evenness Condition}
\cite[Theorem~0.7]{Zie}. This predicts that  the boundary in
Example \ref{ex:morewaves} is subdivided into $12$ curvy triangles
 $B_\mu$, namely those indexed by $\mu \in \{123, 128, 134, 145, 156, 167, 178, 238, 348,458,568,678\}$. Our belief is supported by related results for
 the moment curve, where
 $\mathcal{A} = \{1,2,\ldots,m\}$, due to
Bik, Czapli\'nski and Wageringel \cite{BCW21}.
Their figures show curvy cyclic polytopes in dimension~$3$.

The theory of triangulations of cyclic polytopes \cite{Ram} now suggests an approach to
 unique recovery even when $m<n$.
Each triangulation consists of a certain subset of $\binom{[n]}{m}$.
If our belief is correct, then  this should induce a curvy triangulation
of ${\rm im}(\psi_{\mathcal{A}})$. 
A general point~$c$ in the image is contained in a unique simplex
$B_\nu$ of the triangulation.
There is a unique $z$ in the locus $R_\nu$ with $\psi_\mathcal{A}(z) = c$.
The assignment $c \mapsto z$ serves as a method for unique recovery.

\smallskip

We conclude with a natural generalization of the problem
discussed in this section. Let $\mathcal{K} = \{K_1,\ldots,K_m\}$ be 
a set of centrally
symmetric convex bodies in $\RR^n$.
Each of these defines a norm $|| \, \cdot \, ||_{K_i}$ on $\RR^n$.
The unit ball for that norm is the convex body $K_i$. Consider the map
\begin{equation}
\label{eq:Kmap}
 \psi_{\mathcal{K}} \,:\,\RR^n_{\geq 0} \,\rightarrow \,\Delta_{m-1} \,:\,\,
x \,\mapsto\, \frac{1}{\sum_{j=1}^m || x ||_{K_j}} \cdot 
\bigl(\, ||x||_{K_1},||x||_{K_2}, \ldots,||x||_{K_m}  \bigr). 
\end{equation}

\begin{problem}
Study the image and the fibers of the map $\psi_{\mathcal{K}}$.
Identify the branch loci of $\psi_{\mathcal{K}}$.
\end{problem}

\bigskip
\bigskip

\noindent
\footnotesize 
{\bf Authors' addresses:}

\smallskip

\noindent Hana Mel\'anov\'a,
University of Vienna,
\hfill {\tt hana.melanova@univie.ac.at}

\noindent Bernd Sturmfels,
MPI-MiS Leipzig and UC Berkeley 
\hfill {\tt bernd@mis.mpg.de}

\noindent Rosa Winter,  MPI-MiS Leipzig
\hfill {\tt rosa.winter@mis.mpg.de}

\end{document}